\documentclass[10pt]{article}

\usepackage{amstext}
\usepackage{amssymb}
\usepackage{amsmath}
\usepackage{amsthm}
\usepackage[dvips]{graphicx}        % ¬ Єа®Ї ЄҐв ¤«п ўЄ«озҐ­Ёп аЁбг­Є®ў eps Є®¬ ­¤®© \includegraphicx{XXX.eps}
\graphicspath{{./pict/}}            % Џгвм Є Ї ЇЄҐ б аЁбг­Є ¬Ё
\usepackage{color}
\usepackage[cp1251]{inputenc}
\usepackage[russian,english]{babel}
\theoremstyle{plain}
\newtheorem{theorem}{Theorem}[section]

\newtheorem{prepos}[theorem]{Proposition}
\newtheorem{corol}[theorem]{Corollary}

\theoremstyle{definition}
\newtheorem{definition}[theorem]{Definition}

\newtheorem{remark}[theorem]{Remark}

\renewcommand{\Re}{\operatorname{Re}}

\begin{document}
\vspace{\baselineskip} \thispagestyle{empty}

\title{Circulants and critical points of polynomials}

\author{
Olga Kushel%\thanks{Research was partially supported by
%National Science Foundation of China, grant no. BC0710062.}
\\
%\small Department of Mathematics,\\
\small Shanghai Jiao Tong University\\
%\small 800 Dong Chuan Road, 200240,\\
%\small Shanghai, P.R. China\\
\small Email: {\tt kushel@sjtu.edu.cn}
 \and
Mikhail Tyaglov%\thanks{Shanghai Oriental Scholar. Research was supported by
%Russian Science Foundation, grant no. 14-11-00022.}
\\
%\small Department of Mathematics,\\
\small Shanghai Jiao Tong University\\
\small and\\
%\small School of Mathematics and Mechanics???,\\
\small Far Eastern Federal University\\
%\small 800 Dong Chuan Road, 200240,\\
%\small Shanghai, P.R. China\\
\small Email: {\tt tyaglov@sjtu.edu.cn}}

\date{\small \today}

\maketitle

\abstract{We prove that for any circulant matrix $C$ of size $n\times n$ with the monic characteristic polynomial~$p(z)$,
the spectrum of its $(n-1)\times(n-1)$ submatrix $C_{n-1}$ constructed with first $n-1$ rows and columns of $C$ consists of all critical points of $p(z)$. Using this fact we provide
a simple proof for the Schoenberg conjecture recently proved by R.\,Pereira and S.\,Malamud. We also
prove full generalization of a higher order Schoenberg-type conjecture proposed by M.\,de Bruin and A.\,Sharma and recently proved by W.S.\,Cheung
and T.W.\,Ng. in its original form, i.e. for polynomials whose mass centre of roots equals zero. In this particular case, our inequality is
stronger than it was conjectured by de Bruin and Sharma. Some Schmeisser's-like results on majorization of critical point of
polynomials are also obtained.}

\setcounter{equation}{0}
%%%%%%%%%%%%%%%%%%%%%%%%%%%%%%%%%%%%%%%%%%%%%%%%%%%%%%%%%%%%%%%%%%%%%%%%%%%%%%%%%%%%
\section{Introduction}
%%%%%%%%%%%%%%%%%%%%%%%%%%%%%%%%%%%%%%%%%%%%%%%%%%%%%%%%%%%%%%%%%%%%%%%%%%%%%%%%%%%%

In~\cite{Schoenberg} I.\,Schoenberg conjectured the following statement

\vspace{2mm}

\noindent{\textbf{Schoenberg's conjecture.}} \textit{Let $\lambda_1$, $\lambda_2$, \ldots, $\lambda_n$ be the roots of a polynomial $p$ of degree
$n\geqslant2$ such that $\sum_{j=1}^n\lambda_j=0$, and let $w_1$, $w_2$, \ldots, $w_{n-1}$ be the roots of the derivative $p'$. Then
\begin{equation*}
\sum_{k=1}^{n-1}|w_k|^2\leqslant\dfrac{n-2}{n}\sum_{j=1}^n|\lambda_j|^2,
\end{equation*}
where equality holds if, and only if, all $\lambda_j$ lie on a straight line.}

\vspace{2mm}

Later M.\, de Bruin, K.\,Ivanov, and A.\,Sharma~\cite{deBruin_Ivan_Sharma} and independently Katsoprinakis~\cite{Katsoprinakis}
showed that Schoenberg's conjecture is equivalent to the following inequality
\begin{equation}\label{Schoenberg_ineq}
\sum_{k=1}^{n-1}|w_k|^2\leqslant\dfrac1{n^2}\left|\sum_{j=1}^n\lambda_j\right|^2+\dfrac{n-2}{n}\sum_{j=1}^n|\lambda_j|^2,
\end{equation}
which was conjectured to be true for all complex polynomials with no restrictions on their roots, and equality holds if, and only if, all $\lambda_j$ lie on a straight line. The inequality~\eqref{Schoenberg_ineq}
was proved by R.\,Pereira~\cite{Pereira} and independently by S.\,Malamud~\cite{Malamud0,Malamud}.

In~\cite{deBruin_Sharma} M.\,de Bruin and A.\,Sharma proposed the following higher order Schoenberg-type conjecture.
\vspace{2mm}

\noindent{\textbf{De Bruin and Sharma's conjecture.}} \textit{Let $\lambda_1$, $\lambda_2$, \ldots, $\lambda_n$ be the roots of a polynomial $p$ of degree
$n\geqslant2$ such that $\sum_{j=1}^n\lambda_j=0$, and let $w_1$, $w_2$, \ldots, $w_{n-1}$ be the roots of the derivative $p'$. Then
\begin{equation*}
\sum_{k=1}^{n-1}|w_k|^4\leqslant\dfrac{n-4}{n}\sum_{j=1}^n|\lambda_j|^4+\dfrac{2}{n^2}\left(\sum_{j=1}^n|\lambda_j|^2\right)^2,
\end{equation*}
where equality holds if, and only if, all $\lambda_j$ lie on a straight line passing through the origin of the complex plane.}

\vspace{2mm}

This conjecture was proved by W.\,Cheung and T.\,Ng in~\cite{Cheung_Ng_1} with use their approach of companion matrices (they also give an
alternative proof of Schoenberg's conjecture). Later, in~\cite{Cheung_Ng_2} they generalized this approach by one-rank perturbation technique. Note that in~\cite{Cheung_Ng_2},
W.\,Cheung and T.\,Ng partially rediscover the so-called ``one-rank perturbation method'' developed by Yu. Barkovsky in his PhD thesis~\cite{Bark_thesis} which, unfortunately, was only partially published. Some ideas of this method can be found among
problems in the lecture notes~\cite{Bark_lectures}. From the point of view of this method it is easy to see that, indeed,
the rank one perturbation approach of W.\,Cheung and T.\,Ng generalizes the approach of differentiators used by R.\,Pereiera in~\cite{Pereira}.

S.\,Malamud~\cite{Malamud0,Malamud} proved Schoenberg's conjecture by using another technique (with the same
underlying ideas, in fact, basing on the same one-rank perturbation method). To prove the conjecture he established the fact saying that for any polynomial $p$ of degree $n$,
there exists a normal matrix $A$ with $p(z)=\det(zI-A)$ such that the spectrum of its $(n-1)\times(n-1)$-submatrix $A_{n-1}$ constructed with the first $n-1$ rows and the first $n-1$ columns consists of the roots of the derivative of $p$ (Proposition 4.2 in~\cite{Malamud}).

In this work, we use Barkovsky's one-rank perturbation method (in its simplified Cheung-Ng's form) to prove that
circulant matrices is an example to Malamud's Proposition 4.2, thus giving explicit simple examples of differentiators for a certain
class of matrices. Note that W.\,Cheung and T.\,Ng also constructed a differentiator explicitly. But they started from a diagonal matrix, so their differentiator has a more complicated form, thus as a consequence, to work with such a differentiator is not easy.

It is known~\cite{Davis} that if $p(z)$ is a complex polynomial, then there exists a circulant matrix
\begin{equation}\label{circulant.matrix}
C=
\begin{pmatrix}
c_0    &c_1    &c_2   &\dots &c_{n-2}&c_{n-1}\\
c_{n-1}&c_0    &c_1   &\dots &c_{n-3}&c_{n-2}\\
c_{n-2}&c_{n-1}&c_0   &\dots &c_{n-4}&c_{n-3}\\
\vdots &\vdots &\vdots&\ddots&\vdots &\vdots \\
c_2    &c_3    &c_4   &\dots &c_0    &c_1    \\
c_1    &c_2    &c_3   &\dots &c_{n-1}&c_0    \\
\end{pmatrix}
\end{equation}
whose characteristic polynomial  (up to a constant factor) is $p(z)$. It turned out (Theorem~\ref{Theorem.circulant.derivative}) that the derivative $p'(z)$ is the characteristic polynomial (up to a constant factor) of the submatrix $C_{n-1}$ constructed with the first $n-1$ rows and columns of $C$. Let $\{w_1,w_2,\cdots,w_{n-1}\}$ be the eigenvalues of the matrix $C_{n-1}$. The entries of the matrix $C_{n-1}$ depend linearly~\cite{Davis} on the roots of the polynomial~$p(z)$, so we can use Schur's theorem~\cite{Schur} (Theorem~\ref{Theorem.Schur}) to estimate sums of squares of absolute values of $w_k$ to prove Schoenberg's conjecture. Then applying the same theorem to the matrix $C_{n-1}^2$ we prove the following fact which generalizes Cheung and Ng's theorem (de Briun and Sharma's conjecture).

\begin{theorem}\label{Theorem.shaurma}
Let $\lambda_1$, $\lambda_2$, \ldots, $\lambda_n$ be the roots of a polynomial $p$ of degree
$n\geqslant2$, and let $w_1$, $w_2$, \ldots, $w_{n-1}$ be the roots of its derivative $p'$. Then
\begin{equation}\label{general.shaurma}
\begin{array}{c}
\displaystyle\sum_{k=1}^{n-1}|w_k|^4\leqslant\dfrac{n-6}{n}\sum_{j=1}^n|\lambda_j|^4+\dfrac{1}{n^2}\left(\sum_{j=1}^n|\lambda_j|^2\right)^2+
\dfrac1{n^2}\left|\sum_{j=1}^n\lambda_j^2-\dfrac1{n^2}\left(\sum_{j=1}^n\lambda_j\right)^2\right|^2+\\
\\
\displaystyle+\dfrac{2}{n}\sum_{j=1}^n|\lambda_j|^2\left|\lambda_j+\dfrac1n\sum_{j=1}^n\lambda_j\right|^2-\dfrac4{n^3}\sum_{j=1}^n|\lambda_j|^2\left|\sum_{j=1}^n\lambda_j\right|^2,
\end{array}
\end{equation}
where equality holds if, and only if, all $\lambda_j$ lie on a straight line.
\end{theorem}

For the specific case considered in de Bruin and Sharma's conjecture, we obtain the following
inequality which is stronger than the original conjecture.
\begin{corol}\label{Corollary.shaurma}
Let $\lambda_1$, $\lambda_2$, \ldots, $\lambda_n$ be the roots of a polynomial $p$ of degree
$n\geqslant2$, and let $w_1$, $w_2$, \ldots, $w_{n-1}$ be the roots of its derivative $p'$. If $\displaystyle\sum_{j=1}^n\lambda_j=0$, then
\begin{equation*}
\sum_{k=1}^{n-1}|w_k|^4\leqslant\dfrac{n-4}{n}\sum_{j=1}^n|\lambda_j|^4+\dfrac{1}{n^2}\left(\sum_{j=1}^n|\lambda_j|^2\right)^2+
\dfrac1{n^2}\left|\sum_{j=1}^n\lambda_j^2\right|^2,
\end{equation*}
where equality holds if, and only if, all $\lambda_j$ lie on a straight line passing through the origin of the complex plane.
\end{corol}

It is easy to see that de Bruin and Sharma's conjecture follows from Corollary~\ref{Corollary.shaurma}, since
\begin{equation*}
\left|\sum_{j=1}^n\lambda_j^2\right|\leqslant\sum_{j=1}^n|\lambda_j|^2.
\end{equation*}

Finally, using some fact on singular values of matrices we prove a Schmeisser-like inequality on the absolute values
of roots of derivatives (critical points) of polynomials.
\begin{theorem}\label{Theorem.Cheung_Ng.majorization}
Let $p$ be a polynomial of degree $n\geqslant2$ with zeroes $\lambda_j$, $j=1,\ldots,n$, and
critical points~$w_k$, $k=1,\ldots,n-1$. If $q$ be a polynomial with zeroes $|\lambda_j|^2$, $j=1,\ldots,n$, and
critical points $\eta_k$, $k=1,\ldots,n-1$, then
%w
\begin{equation}\label{Cheung_Ng.majorization.like}
(\Phi(|w_1|),\ldots,\Phi(|w_{n-1}|))\prec_{w}(\Phi(\sqrt{\eta_1}),\ldots,\Phi(\sqrt{\eta_{n-1}})),
\end{equation}
for any increasing function $\Phi:[0,+\infty)\mapsto\mathbb{R}$ such that $\Phi\circ\exp$ is convex on $\mathbb{R}$.
\end{theorem}
\noindent Here the symbol $\prec_{w}$ means that the set $\{\sqrt{\eta_1},\ldots,\sqrt{\eta_{n-1}}\}$ weakly majorizes the set $\{|w_1|,\ldots,|w_{n-1}|\}$ (for definition of majorization see Section~\ref{section:majorization}).

For the case when $p(0)=0$ this theorem follows from~\cite[Theorem 2.1]{Cheung_Ng_1}, but for an arbitrary
complex polynomials the estimate~\eqref{Cheung_Ng.majorization.like} seems to be the best known.

As a by-product, we obtain the following curious statement.
\begin{theorem}\label{Theorem.critical.squares}
Let $p(z)$ be a real polynomials with \textbf{positive} roots $\lambda_j$, $j=1,\ldots,n$, and let $\xi_k$, $k=1,\ldots,n-1$,
be its critical points. If $\eta_k$, $k=1,\ldots,n-1$, are the critical points of the polynomial $q$ with roots $\lambda_j^2$, then
\begin{equation*}\label{crit.point.majorization}
(\Phi(\xi_1),\ldots,\Phi(\xi_{n-1}))\prec_{w}(\Phi(\sqrt{\eta_1}),\ldots,\Phi(\sqrt{\eta_{n-1}})),
\end{equation*}
for any increasing function $\Phi:[0,+\infty)\mapsto\mathbb{R}$ such that $\Phi\circ\exp$ is convex on $\mathbb{R}$.
\end{theorem}

The paper is organized as follows. In Section~\ref{Section:circulants} we  remind to the reader some necessary facts on spectra of circulants and prove our main Theorem~\ref{Theorem.circulant.derivative}. Section~\ref{Section:Conjectures} is devoted
to the proof of Theorem~\ref{Theorem.shaurma} on circulants. In Section~\ref{section:majorization} we discuss the usefulness of circulants
for proving theorems on majorization of the absolute values of critical points of polynomials. In particular, we prove Theorems~\ref{Theorem.Cheung_Ng.majorization} and~\ref{Theorem.critical.squares}. The last Section~\ref{Section:open.questions} we discuss other possibilities to use Theorem~\ref{Theorem.circulant.derivative} for localization of roots of polynomials and pose some problems that can be of interest in this sense.

\setcounter{equation}{0}
%%%%%%%%%%%%%%%%%%%%%%%%%%%%%%%%%%%%%%%%%%%%%%%%%%%%%%%%%%%%%%%%%%%%%%%%%%%%%%%%%%%%
\section{Circulants}\label{Section:circulants}
%%%%%%%%%%%%%%%%%%%%%%%%%%%%%%%%%%%%%%%%%%%%%%%%%%%%%%%%%%%%%%%%%%%%%%%%%%%%%%%%%%%%

A matrix of the form
\begin{equation}%\label{circulant.matrix}
C=
\begin{pmatrix}
c_0    &c_1    &c_2   &\dots &c_{n-2}&c_{n-1}\\
c_{n-1}&c_0    &c_1   &\dots &c_{n-3}&c_{n-2}\\
c_{n-2}&c_{n-1}&c_0   &\dots &c_{n-4}&c_{n-3}\\
\vdots &\vdots &\vdots&\ddots&\vdots &\vdots \\
c_2    &c_3    &c_4   &\dots &c_0    &c_1    \\
c_1    &c_2    &c_3   &\dots &c_{n-1}&c_0    \\
\end{pmatrix}
\end{equation}
is called \textit{circulant}. The polynomial
\begin{equation*}
f(x)=c_0+c_1x+c_2x^2+\dots+c_{n-1}x^{n-1}
\end{equation*}
is called the \textit{associated polynomial} of the matrix $C$. It is very well-known~\cite{Davis} that the eigenvalues
$\lambda_j$, $j=1,\ldots,n$, of the matrix $C$ can be expressed as follows
\begin{equation}\label{circulant.eigenvalues}
\lambda_j=c_0+c_1\omega_{j-1}+c_2\omega_{j-1}^2+\dots+c_{n-1}\omega_{j-1}^{n-1}=f(\omega_{j-1}),\quad j=1,\ldots,n,
\end{equation}
where
\begin{equation*}
\omega_{k}=e^{\tfrac{2\pi i}{n}k},\quad k=0,1,\ldots,n-1,
\end{equation*}
are the $n^{th}$ roots of unity. The formula~\eqref{circulant.eigenvalues} shows that for
any set $\{\lambda_1,\ldots,\lambda_n\}$ of $n$ complex numbers, there exist at most $n!$
(number of permutations) different circulant matrices with the spectrum at this numbers
(counting multiplicities). At the same time, if we fix the arrangement of these numbers%\footnote{That is, we assign an index
%to each number.}
, for example, if we arrange the numbers $\lambda_j$ as follows
\begin{equation}\label{order.EVal}
|\lambda_1|\geqslant|\lambda_2|\geqslant\cdots\geqslant|\lambda_n|,
\end{equation}
then there exists exactly one circulant matrix whose spectrum $\{\lambda_1,\ldots,\lambda_n\}$ satisfies
the equations~\eqref{circulant.eigenvalues} and have the prescribed arrangement (e.g. satisfy the condition~\eqref{order.EVal}). This follows from the simple fact that the system~\eqref{circulant.eigenvalues}
is not singular, since its matrix
\begin{equation}\label{Vandermonde}
\begin{pmatrix}
1         &1         &\dots&1             &1             \\
\omega_  0&\omega_1  &\dots&\omega_{n-2}  &\omega_{n-1}  \\
\omega_0^2&\omega_1^2&\dots&\omega_{n-2}^2&\omega_{n-1}^2\\
\vdots    &\vdots    &\ddots&\vdots       &\vdots        \\
\omega_0^{n-2}&\omega_1^{n-2}&\dots&\omega_{n-2}^{n-2}&\omega_{n-1}^{n-2}\\
\omega_0^{n-1}&\omega_1^{n-1}&\dots&\omega_{n-2}^{n-1}&\omega_{n-1}^{n-1}\\
\end{pmatrix}
\end{equation}
has determinant equal to $n$. With a fixed arrangement of eigenvalues, to the eigenvalue $\lambda_j$ there corresponds
the following eigenvector
\begin{equation}\label{circulant.eigenvectors}
u_j=
\begin{pmatrix}
1\\
\omega_{j-1}\\
\omega_{j-1}^2\\
\vdots\\
\omega_{j-1}^{n-1}\\
\end{pmatrix}\ ,
\end{equation}
so that
$$
Cu_j=\lambda_ju_j,\qquad j=1,\ldots,n.
$$
Moreover,
\begin{equation}\label{orthogonality.eigenvectors}
\langle u_k,u_j\rangle=n\delta_{kj}, \qquad k,j=1,\ldots,n,
\end{equation}
where $\langle\cdot,\cdot\rangle$ is the inner product in $\mathbb{C}^n$, and $\delta_{k,j}$ is the Kronecker symbol.

Let us denote by $C_{n-1}$ the following $(n-1)\times(n-1)$ matrix constructed with first $n-1$ rows and columns
of the circulant $C$:
\begin{equation}\label{circulant.SUBmatrix}
C_{n-1}=
\begin{pmatrix}
c_0    &c_1    &c_2   &\dots &c_{n-3}&c_{n-2}\\
c_{n-1}&c_0    &c_1   &\dots &c_{n-4}&c_{n-3}\\
c_{n-2}&c_{n-1}&c_0   &\dots &c_{n-5}&c_{n-4}\\
\vdots &\vdots &\vdots&\ddots&\vdots &\vdots \\
c_3    &c_4    &c_5   &\dots &c_0    &c_1    \\
c_2    &c_3    &c_4   &\dots &c_{n-1}&c_0    \\
\end{pmatrix}
\end{equation}

Now we are in a position to establish one of the main results of the present work revealing an interesting and deep
property of circulant matrices.

\begin{theorem}\label{Theorem.circulant.derivative}
For a given complex circulant matrix $C$ with the monic characteristic polynomial~$p(z)$, the spectrum
of its submatrix $C_{n-1}$ coincides with the set of all critical points of $p(z)$ (counting multiplicities).
\end{theorem}
\begin{proof}
Let $\lambda_j$, $j=1,\ldots,n$, be the roots of the polynomial $p(z)$. Suppose first that
$$
p'(\lambda_j)\neq0\ \ \ \text{and}\ \ \ \lambda_j\neq0,\qquad j=1,\ldots,n,
$$
and construct the following rational function
$$
R(z)=\dfrac1n\dfrac{zp'(z)}{p(z)}=1+\dfrac{1}{n}\sum_{j=1}^n\dfrac{\lambda_j}{z-\lambda_j}.
$$

Let $u$ be the vector as follows
\begin{equation}\label{vector.u}
u\stackrel{def}{=}\dfrac1n\sum_{j=1}^nu_j=
\dfrac1n\sum_{j=1}^n
\begin{pmatrix}
1\\
\omega_{j-1}\\
\omega_{j-1}^2\\
\vdots\\
\omega_{j-1}^{n-1}\\
\end{pmatrix}=
\begin{pmatrix}
1\\
0\\
0\\
\vdots\\
0\\
\end{pmatrix}\ ,
\end{equation}
where $u_j$, $j=1,\ldots,n$, are the eigenvectors of the matrix $C$ defined in~\eqref{circulant.eigenvectors},
and let $H$ be the following one-rank matrix
\begin{equation*}\label{matrix.H}
H\stackrel{def}{=}u\otimes u=\begin{pmatrix}
1&0&\dots&0&0\\
0&0&\dots&0&0\\
\vdots&\vdots&\ddots&\vdots&\vdots\\
0&0&\dots&0&0\\
0&0&\dots&0&0\\
\end{pmatrix}\ ,
\end{equation*}
thus $\forall x\in\mathbb{C}^n$,
\begin{equation*}
Hx=\langle x,u\rangle u,
\end{equation*}
so that
$$
Hu_j=u,\quad j=1,\ldots,n,
$$
as it follows from~\eqref{vector.u} and~\eqref{orthogonality.eigenvectors} (or from the definition of the matrix $H$). Also
it is clear that $Hu=u$, so $H^2=H$, and the operator
$$
P=I-H
$$
is a projector to the subspace
$$
\{x\in\mathbb{C}^n\ :\ \langle x,u\rangle=0\}.
$$
as a difference of the identity operator and the projector $H$ to the subspace $\mathrm{span}\{u\}$.

We now claim that $zp'(z)=n\det(zI-PC)$. Indeed, since
$$
(zI-C)^{-1}u_j=\dfrac{u_j}{z-\lambda_j}, \quad j=1,\ldots,n,
$$
and
$$
HCu_j=\lambda_ju,
$$
one has for $z\neq\lambda_j$, $j=1,\ldots,n$,
$$
HC(zI-C)^{-1}u=\dfrac1n\sum_{j=1}^n\dfrac{HCu_j}{z-\lambda_j}=\dfrac1n\sum_{j=1}^n\dfrac{\lambda_j}{z-\lambda_j}u=
\left(R(z)-1\right)u.
$$
From this identity it is clear that the matrix $HC(zI-C)^{-1}$ has the form
$$
HC(zI-C)^{-1}=
\begin{pmatrix}
R(z)-1&\times&\dots&\times&\times\\
0&0&\dots&0&0\\
\vdots&\vdots&\ddots&\vdots&\vdots\\
0&0&\dots&0&0\\
0&0&\dots&0&0\\
\end{pmatrix}\ ,
$$
so the only (possibly) nonzero entries of this matrix lie on the first row. Thus, we have
$$
\det(zI-PC)=\det([zI-C]+HC)=\det(zI-C)\det\left(I+HC(zI-C)^{-1}\right)=p(z)(R(z)-1)=\dfrac{1}{n}zp'(z),
$$
so
$$
zp'(z)=n\det(zI-PC).
$$
On the other hand,
$$
n\det(zI-PC)=nz\det(zI-C_{n-1}),
$$
where $C_{n-1}$ is defined in~\eqref{circulant.SUBmatrix}. Therefore,
\begin{equation*}\label{circulant.SUBmatrix.det}
p'(z)=n\det(zI-C_{n-1}).
\end{equation*}

If $p(z)$ has simple roots one of which is equal to zero, the we consider a polynomial $\widehat{p}(z)=p(z-\varepsilon)$, $\varepsilon\neq0$,
so that $\widehat{p}'(z)=p'(z-\varepsilon)$. Then we find for $\widehat{p}(z)$ a circulant $\widehat{C}$ such that
$\widehat{p}=\det(zI-\widetilde{C})$ and $\widehat{p}'(z)=n\det(zI-\widehat{C}_{n-1})$. Changing variables $z\to z+\varepsilon$
gives us $p(z)=\det(zI-C)$, where $C=\widehat{C}-\varepsilon I$, and $p'(z)=n\det(zI-C_{n-1})$.

Finally, if $p(z)$ has multiple roots, then one can approximate $p(z)$ by a sequence of polynomials~$p_m(z)$ with simple roots,
so $p_m(z)\to p(z)$ and $p'_m(z)\to p'(z)$ as $m\to+\infty$ uniformly on compact sets. Clearly, in this case the corresponding
circulants $C^{(m)}$ such that $p_m(z)=\det(zI-C^{(m)})$ tend to a circulant matrix $C$ such that $p(z)=\det(zI-C)$ and $p'(z)=n\det(zI-C_{n-1})$, as required.
\end{proof}
\begin{remark}
In the proof of Theorem~\ref{Theorem.circulant.derivative}, we used the technique from the work~\cite{Cheung_Ng_2} by Cheung and Ng,
but, indeed, the fact proved in Theorem~\ref{Theorem.circulant.derivative} is a simple consequence of a more general fact
established in~\cite{Bark_thesis} and~\cite[Section 2]{Bark_Kushel_Tyaglov}.
\end{remark}

The circulant matrix $C$ is normal~\cite{Davis}, that is,
\begin{equation*}
CC^*=C^*C.
\end{equation*}
Here the matrix $C^*$ adjoint to $C$ is also a circulant. Its eigenvalues are $\overline{\lambda_j}$ with the same corresponding
eigenvectors $u_j$ defined in~\eqref{circulant.eigenvectors}. Let us denote $A\stackrel{def}{=}CC^*$. The matrix $A$ is a circulant
as well
\begin{equation}\label{circulant.A}
A=
\begin{pmatrix}
a_0    &a_1    &a_2   &\dots &a_{n-2}&a_{n-1}\\
a_{n-1}&a_0    &a_1   &\dots &a_{n-3}&a_{n-2}\\
a_{n-2}&a_{n-1}&a_0   &\dots &a_{n-4}&a_{n-3}\\
\vdots &\vdots &\vdots&\ddots&\vdots &\vdots \\
a_2    &a_3    &a_4   &\dots &a_0    &a_1    \\
a_1    &a_2    &a_3   &\dots &a_{n-1}&a_0    \\
\end{pmatrix}\ ,
\end{equation}
where
\begin{equation}\label{coeff.a_0}
a_0=\sum_{j=0}^{n-1}|c_j|^2
\end{equation}
and
\begin{equation}\label{coeff.a_k}
a_k=\sum_{j=0}^{k-1}c_j\overline{c}_{n-k+j}+\sum_{j=0}^{n-k-1}c_{k+j}\overline{c}_j.
\end{equation}
Here bar means the complex conjugation. Clearly,
\begin{equation}
a_k=\overline{a}_{n-k},\quad k=1,\ldots,n-1,\qquad\text{and}\qquad a_0\in\mathbb{R}.
\end{equation}
The eigenvalues of the matrix $A$ are $|\lambda_j|^2$, $j=1,\ldots,n$.

Generally speaking, the matrix $C_{n-1}$ is not a circulant and is not normal. However, it is possible to describe the case
when it is normal. The following theorem is a consequence of a theorem due to Ky Fan and G.\,Pall~\cite[Theorem~2]{Ky_Fan_Pall} which deals with arbitrary normal matrices and its principal submatrices.
\begin{theorem}\label{Theorem.normal.SUBmatrix}
The submatrix $C_{n-1}$ of a circulant $C$ is normal if, and only if, all the eigenvalues
of~$C$ lie on a straight line.
\end{theorem}
This theorem also follows from~\cite[Proposition 3.4]{Pereira}, since $C_{n-1}$ is a differentiator. However, we prove Theorem~\ref{Theorem.normal.SUBmatrix} here in order to reveal some specific properties of circulants and their submatrices.
To prove it, we need the following simple fact whose proof we provide here for completeness.
\begin{prepos}\label{Proposition.real.EV}
The circulant $C$ has all eigenvalues \textbf{real} if, and only if, it is self-adjoint or, equivalently, if its coefficients
satisfy the reflection property
\begin{equation}\label{reflection}
c_0\in\mathbb{R},\qquad c_{n-k}=\overline{c_k},\qquad k=1,\ldots,n.
\end{equation}
\end{prepos}
\begin{proof}
It is clear that $C$ is self-adjoint if, and only if, the reflection property~\eqref{reflection} holds. So, this property
implies the reality of the spectrum of $C$.

Conversely, if all the eigenvalues $\lambda_j$ of $C$ are real, then from~\eqref{circulant.eigenvalues} one has
\begin{equation*}
\overline{\lambda_j}=\overline{c_0}+\overline{c_{n-1}}\omega_{j-1}+\overline{c_{n-2}}\omega_{j-1}^2
+\dots+\overline{c_{1}}\omega_{j-1}^{n-1}=\lambda_j,\quad j=1,\ldots,n.
\end{equation*}
Since the matrix~\eqref{Vandermonde} of this system is non-singular, the solution is unique (since the arrangement of $\lambda_j$ is supposed to be fixed), so it must
satisfy the condition~\eqref{reflection}.
\end{proof}

If $C$ has only real eigenvalues, so it is self-adjoint, then its (principal) submatrix $C_{n-1}$ is also self-adjoint, and
has only real eigenvalues. The reality of the eigenvalues of $C_{n-1}$ also follows from Theorem~\ref{Theorem.circulant.derivative} but
it does not guarantee, generally speaking, that $C_{n-1}$ is self-adjoint.

\begin{proof}[Proof of Theorem~\ref{Theorem.normal.SUBmatrix}]
If all $\lambda_j$ lie on a straight line in the complex plane, then they have the form
$$
\lambda_j=\alpha\eta_j+\beta,\quad j=1,\ldots,n,
$$
where $\alpha,\beta\in\mathbb{C}$, and $\eta_j\in\mathbb{R}$. Therefore, the matrix $C$ can be represented
as follows
$$
C=\alpha\widetilde{C}+\beta I,
$$
where $\widetilde{C}$ is a circulant matrix whose eigenvalues are the numbers $\eta_j$. By Theorem~\ref{Proposition.real.EV},
$\widetilde{C}$ is self-adjoint. Consequently, $C_{n-1}$ has the form
\begin{equation}\label{normal.matrix.form}
C_{n-1}=\alpha\widetilde{C}_{n-1}+\beta I_{n-1},
\end{equation}
where $I_{n-1}$ is the $(n-1)\times(n-1)$ identity matrix, and $\widetilde{C}_{n-1}$ is self-adjoint as principal
submatrix of~$\widetilde{C}$. It is easy to check now that $C_{n-1}$ is normal.

Conversely, suppose that $C_{n-1}$ is normal. Let us denote
\begin{equation}\label{matrix.B}
B\stackrel{def}{=}C_{n-1}C^*_{n-1}.
\end{equation}
The matrix $B$ is self-adjoint, and its entries have the form
\begin{equation}\label{matrix.B.entries}
b_{lk}=a_{k-l}-c_{n-l}\overline{c}_{n-k},\qquad 1\leqslant l\leqslant k\leqslant n-1.
\end{equation}
At the same time, the entries $\widetilde{b}_{kl}$ of the self-adjoint matrix
\begin{equation}\label{matrix.B.wave}
\widetilde{B}\stackrel{def}{=}C^*_{n-1}C_{n-1}.
\end{equation}
have the form
\begin{equation}\label{matrix.B.wave.entries}
\widetilde{b}_{lk}=a_{k-l}-c_{k}\overline{c}_{l},\qquad 1\leqslant l\leqslant k\leqslant n-1.
\end{equation}
Thus, from normality of the matrix $C_{n-1}$ it follows that $B=\widetilde{B}$. This equality together with~\eqref{matrix.B.entries}--\eqref{matrix.B.wave.entries} implies
\begin{equation}\label{equalities.1}
|c_{k}|=|c_{n-k}|,\quad k=1,\ldots,n-1,
\end{equation}
and
\begin{equation}\label{equalities.2}
c_{n-l}\overline{c}_{n-k}=\overline{c}_{l}c_{k},\qquad 1\leqslant l< k\leqslant n-1.
\end{equation}
If we put $c_k=r_ke^{i\varphi_k}$, $r_k\in\mathbb{R}$, $\varphi_k\in[0,2\pi)$, $k=1,\ldots,n-1$, then the
identities~\eqref{equalities.1}--\eqref{equalities.2} imply
\begin{equation*}
\varphi_k+\varphi_{n-k}=\varphi_{l}+\varphi_{n-l},\qquad 1\leqslant l< k\leqslant n-1.
\end{equation*}
Denoting $2\theta\stackrel{def}{=}\varphi_k+\varphi_{n-k}$ for all $k=1,\ldots,n-1$, and
$$
\widehat{\varphi}_{k}\stackrel{def}{=}-\theta+\varphi_k,\qquad k=1,\ldots,n-1,
$$
we get
$$
\widehat{\varphi}_k+\widehat{\varphi}_{n-k}=-2\theta+\varphi_k+\varphi_{n-k}=0,\qquad k=1,\ldots,n-1,
$$
so the numbers $\widehat{c}_{k}\stackrel{def}{=}c_ke^{-i\theta}$ satisfy the following conditions
\begin{equation}\label{reflection.hat}
\widehat{c}_{k}=\overline{\widehat{c}_{n-k}},\qquad k=1,\ldots,n-1.
\end{equation}
This means that we can represent the matrix $C_{n-1}$ as follows
\begin{equation*}
C_{n-1}=e^{i\theta}\widehat{C}_{n-1}+c_0I_{n-1},
\end{equation*}
where the matrix
\begin{equation*}
\widehat{C}_{n-1}=
\begin{pmatrix}
0    &\widehat{c}_1    &\widehat{c}_2   &\dots &\widehat{c}_{n-3}&\widehat{c}_{n-2}\\
\widehat{c}_{n-1}&0    &\widehat{c}_1   &\dots &\widehat{c}_{n-4}&\widehat{c}_{n-3}\\
\widehat{c}_{n-2}&\widehat{c}_{n-1}&0   &\dots &\widehat{c}_{n-5}&\widehat{c}_{n-4}\\
\vdots &\vdots &\vdots&\ddots&\vdots &\vdots \\
\widehat{c}_3    &\widehat{c}_4    &\widehat{c}_5   &\dots &0    &\widehat{c}_1    \\
\widehat{c}_2    &\widehat{c}_3    &\widehat{c}_4   &\dots &\widehat{c}_{n-1}&0    \\
\end{pmatrix}
\end{equation*}
is self-adjoint by~\eqref{reflection.hat}. Thus the initial matrix $C$ has the form
$$
C=e^{i\theta}\widehat{C}+c_{0}I,
$$
where $\widehat{C}$ is self-adjoint. It is clear now that all eigenvalues of $C$ lie on a straight line.
\end{proof}

For completeness we prove here the following simple and curious fact. The usefulness of this statement is discussed in Section~\ref{Section:open.questions}.
\begin{prepos}\label{Proposition.abs.val}
The matrix $\sqrt{CC^*}$ is a (unique) circulant matrix whose eigenvalues are $|\lambda_j|$, $j=1,\ldots,n$ for a fixed arrangement of $|\lambda_j|$.
\end{prepos}
\begin{proof}
As we mentioned above, there exists a unique circulant $\widehat{C}$ whose eigenvalues are $|\lambda_j|$ with a fixed arrangement, say, satisfying
the inequalities~\eqref{order.EVal}. By Proposition~\ref{Proposition.real.EV}, $\widehat{C}$ is self-adjoint, so it is positive semidefinite self-adjoint matrix.
It is clear that $\widehat{C}^2=CC^*$, since the eigenvalues of the circulants $\widehat{C}^2$ and $CC^*$ coincide up to their arrangement. The positive semidefinite square
root of the positive semidefinite self-adjoint matrix $CC^*$ is unique~\cite[Theorem~7.2.6]{Horn_Johnson}, therefore $\widehat{C}=\sqrt{CC^*}$, as required.
\end{proof}
%

%Finally, let us recall a few fact to be of use in the next Sections. The following theorem due to Schur~\cite{Schur}.
Finally, let us recall the following theorem due to Schur~\cite{Schur} which is to be of use in the next Section.
\begin{theorem}\label{Theorem.Schur}
If $T=(t_{kj})_{k,j=1}^n$ is an $n\times n$ matrix with eigenvalues $\mu_1$, $\mu_2$, \ldots, $\mu_n$, then
\begin{equation*}\label{Schur.inequalities}
\sum_{j=1}^{n}|\mu_j|^2\leqslant Tr(TT^*)=\sum_{k,j=1}^n|t_{kj}|^2,
\end{equation*}
where equality holds if, and only if, $A$ is normal.
\end{theorem}
%

%In Section~\ref{section:majorization} we will need a relation between eigenvalues and singular values of a matrix.
%Recall that singular values of a matrix $A$ are the eigenvalues of the matrix $\sqrt{AA^*}$. The following theorem
%can be found in~\cite[Theorem~3.3.13]{Horn_Johnson}.
%
%\begin{theorem}
%Let an $n\times n$ matrix $A$ have ordered singular values $\sigma_1(A)\geqslant\ldots\geqslant\sigma_n(A)\geqslant0$
%and eigenvalues $\{\lambda_1(A),\ldots,\lambda_n(A)\}$ ordered so that $|\lambda_1(A)|\geqslant\ldots\geqslant|\lambda_n(A)|\geqslant0$.
%Then
%
%\begin{equation}\label{sing.val.majorization}
%(\Phi(|\lambda_1(A)|),\ldots,\Phi(|\lambda_n(A)|))\prec_{w}(\Phi(\sqrt{\xi_1}),\ldots,\Phi(\sqrt{\xi_{n-1}})),
%\end{equation}
%
%for any increasing function $\Phi:[0,+\infty)\mapsto\mathbb{R}$ such that $\Phi\circ\exp$ is convex on $\mathbb{R}$.
%\end{theorem}
%

\setcounter{equation}{0}
%%%%%%%%%%%%%%%%%%%%%%%%%%%%%%%%%%%%%%%%%%%%%%%%%%%%%%%%%%%%%%%%%%%%%%%%%%%%%%%%%%%%
\section{Schoenberg's and de Bruin and Sharma's conjectures}\label{Section:Conjectures}
%%%%%%%%%%%%%%%%%%%%%%%%%%%%%%%%%%%%%%%%%%%%%%%%%%%%%%%%%%%%%%%%%%%%%%%%%%%%%%%%%%%%

In this Section we show that Schoenberg's conjecture is a simple consequence of Theorem~\ref{Theorem.Schur} and~\ref{Theorem.circulant.derivative}
as well as de Bruin and Sharma's conjecture which we generalize here.

\begin{theorem}[Pereira~\cite{Pereira}, Malamud~\cite{Malamud}]
Let $\lambda_1$, $\lambda_2$, \ldots, $\lambda_n$ be the roots of a polynomial $p$ of degree
$n\geqslant2$, and let $w_1$, $w_2$, \ldots, $w_{n-1}$ be the roots of the derivative $p'$. Then
\begin{equation*}
\sum_{k=1}^{n-1}|w_k|^2\leqslant\dfrac1{n^2}\left|\sum_{j=1}^n\lambda_j\right|^2+\dfrac{n-2}{n}\sum_{j=1}^n|\lambda_j|^2,
\end{equation*}
where equality holds if, and only if, all $\lambda_j$ lie on a straight line.
\end{theorem}
\begin{proof}
Let $C$ be a circulant matrix whose eigenvalues are $\lambda_j$, $j=1,\ldots,n$. From Theorems~\ref{Theorem.circulant.derivative} and~\ref{Theorem.Schur}, we have
\begin{equation}\label{sum.square.deriv.roots}
\sum_{k=1}^{n-1}|w_k|^2\leqslant Tr(C_{n-1}C_{n-1}^*)=(n-1)|c_0|^2+(n-2)\sum_{k=1}^{n-1}|c_k|^2=|c_0|^2+(n-2)\sum_{k=0}^{n-1}|c_k|^2.
\end{equation}
On the other hand, by normality of the circulant $C$, one has
\begin{equation}\label{sum.c_k}
\sum_{j=1}^{n}|\lambda_j|^2=n\sum_{k=0}^n|c_k|^2.
\end{equation}
Moreover,
\begin{equation}\label{c_0}
c_0=\dfrac1n\sum_{j=1}^{n}\lambda_j.
\end{equation}
Now~\eqref{sum.square.deriv.roots}--\eqref{c_0} together with Theorem~\ref{Theorem.Schur} imply the theorem.
\end{proof}

We are now in a position to prove the full generalization of de Bruin and Sharma's conjecture.

\vspace{2mm}

\noindent\textbf{Theorem 1.1.}\textit{
Let $\lambda_1$, $\lambda_2$, \ldots, $\lambda_n$ be the roots of a polynomial $p$ of degree
$n\geqslant2$, and let $w_1$, $w_2$, \ldots, $w_{n-1}$ be the roots of its derivative $p'$. Then
\begin{equation}\label{general.shaurma}
\begin{array}{c}
\displaystyle\sum_{k=1}^{n-1}|w_k|^4\leqslant\dfrac{n-6}{n}\sum_{j=1}^n|\lambda_j|^4+\dfrac{1}{n^2}\left(\sum_{j=1}^n|\lambda_j|^2\right)^2+
\dfrac1{n^2}\left|\sum_{j=1}^n\lambda_j^2-\dfrac1{n^2}\left(\sum_{j=1}^n\lambda_j\right)^2\right|^2+\\
\\
\displaystyle+\dfrac{2}{n}\sum_{j=1}^n|\lambda_j|^2\left|\lambda_j+\dfrac1n\sum_{j=1}^n\lambda_j\right|^2-\dfrac4{n^3}\sum_{j=1}^n|\lambda_j|^2\left|\sum_{j=1}^n\lambda_j\right|^2,
\end{array}
\end{equation}
where equality holds if, and only if, all $\lambda_j$ lie on a straight line.
}
\begin{proof}
Since the eigenvalues of the matrix $C_{n-1}^2$ are $w_k^2$, $k=1,\ldots,n-1$, from Theorem~\ref{Theorem.Schur}
we have
\begin{equation*}\label{shaurma.estimate}
\sum_{k=1}^{n-1}|w_k|^4\leqslant Tr\left(C_{n-1}C_{n-1}C_{n-1}^*C_{n-1}^*\right)=Tr(C_{n-1}C_{n-1}^*C_{n-1}^*C_{n-1})=Tr(B\widetilde{B}),
\end{equation*}
where $B$ and $\widetilde{B}$ are defined in~\eqref{matrix.B}--\eqref{matrix.B.entries} and
in~\eqref{matrix.B.wave}--\eqref{matrix.B.wave.entries}, respectively. To calculate $Tr(B\widetilde{B})$, let us note first that
\begin{equation*}
B=A_{n-1}-D\qquad\text{and}\qquad \widetilde{B}=A_{n-1}-\widetilde{D},
\end{equation*}
where $A_{n-1}$ is the $(n-1)^{th}$ leading principal submatrix of the circulant $A$ defined in~\eqref{circulant.A}, while
\begin{equation*}\label{matrix.D}
D=\begin{pmatrix}
|c_{n-1}|^2    &c_{n-1}\overline{c}_{n-2}    &c_{n-1}\overline{c}_{n-3}   &\dots &c_{n-1}\overline{c}_{2}&c_{n-1}\overline{c}_1\\
c_{n-2}\overline{c}_{n-1}&|c_{n-2}|^2    &c_{n-2}\overline{c}_{n-3}   &\dots &c_{n-2}\overline{c}_2&c_{n-2}\overline{c}_1\\
c_{n-3}\overline{c}_{n-1}&c_{n-3}\overline{c}_{n-2}&|c_{n-3}|^2   &\dots &c_{n-3}\overline{c}_2&c_{n-3}\overline{c}_1\\
\vdots &\vdots &\vdots&\ddots&\vdots &\vdots \\
c_2\overline{c}_{n-1}    &c_2\overline{c}_{n-2}    &c_2\overline{c}_{n-3}   &\dots &|c_2|^2    &c_2\overline{c}_{1}    \\
c_1\overline{c}_{n-1}    &c_1\overline{c}_{n-2}    &c_1\overline{c}_{n-3}   &\dots &c_{1}\overline{c}_{2}&|c_1|^2    \\
\end{pmatrix}
\end{equation*}
and
\begin{equation*}\label{matrix.D.wave}
\widetilde{D}=
\begin{pmatrix}
|c_1|^2    &c_2\overline{c}_{1}    &c_3\overline{c}_{1}   &\dots &c_{n-2}\overline{c}_{1}&c_{n-1}\overline{c}_{1}\\
c_{1}\overline{c}_{2}&|c_2|^2    &c_3\overline{c}_{2}   &\dots &c_{n-2}\overline{c}_{2}&c_{n-1}\overline{c}_{2}\\
c_{1}\overline{c}_{3}&c_{2}\overline{c}_{3}&|c_3|^2   &\dots &c_{n-2}\overline{c}_{3}&c_{n-1}\overline{c}_{3}\\
\vdots &\vdots &\vdots&\ddots&\vdots &\vdots \\
c_1\overline{c}_{n-2}    &c_2\overline{c}_{n-2}    &c_3\overline{c}_{n-2}   &\dots &|c_{n-2}|^2    &c_{n-1}\overline{c}_{n-2}    \\
c_1\overline{c}_{n-1}    &c_2\overline{c}_{n-1}    &c_3\overline{c}_{n-1}   &\dots &c_{n-2}\overline{c}_{n-1}&|c_{n-1}|^2    \\
\end{pmatrix}
\end{equation*}
Since $A$ is self-adjoint, $A_{n-1}$ is self-adjoint, as well. Thus we have
\begin{equation*}
Tr(B\widetilde{B})=Tr(A_{n-1}^2)+Tr(D\widetilde{D})-Tr(A_{n-1}\widetilde{D})-Tr(A_{n-1}D).
\end{equation*}
It is easy to see that
\begin{equation*}
Tr(A_{n-1}^2)=a_0^2+(n-2)\sum_{k=0}^{n-1}|a_k|^2=a_0^2+(n-2)e_0,
\end{equation*}
where
$$
e_0=\sum_{k=0}^{n-1}|a_k|^2
$$
is the diagonal entry of the circulant matrix
\begin{equation*}\label{circulant.E}
E\stackrel{def}{=}AA^*=
\begin{pmatrix}
e_0    &e_1    &e_2   &\dots &e_{n-2}&e_{n-1}\\
e_{n-1}&e_0    &e_1   &\dots &e_{n-3}&e_{n-2}\\
e_{n-2}&e_{n-1}&e_0   &\dots &e_{n-4}&e_{n-3}\\
\vdots &\vdots &\vdots&\ddots&\vdots &\vdots \\
e_2    &e_3    &e_4   &\dots &e_0    &e_1    \\
e_1    &e_2    &e_3   &\dots &e_{n-1}&e_0    \\
\end{pmatrix}\ ,
\end{equation*}
whose eigenvalues are $|\lambda_j|^4$, $j=1,\ldots,n$, so
\begin{equation}\label{e_0}
e_0=\dfrac1n\sum_{j=1}^n|\lambda_j|^4.
\end{equation}

It is clear that
\begin{equation*}
Tr(D\widetilde{D})=\left|\sum_{k=1}^{n-1}c_kc_{n-k}\right|^2.
\end{equation*}

Furthermore, for $Tr(A_{n-1}D)$ we have by~\eqref{coeff.a_0}--\eqref{coeff.a_k}
\begin{equation*}
\begin{array}{c}
\displaystyle Tr(A_{n-1}D)=a_0\sum_{k=1}^{n-1}|c_{n-k}|^2+\sum_{l=2}^{n-1}\overline{c}_{n-l}\sum_{k=1}^{l-1}c_{n-k}a_{n-l+k}+\sum_{l=1}^{n-2}\overline{c}_{n-l}\sum_{k=l+1}^{n-1}c_{n-k}a_{k-l}=\\
\\
\displaystyle =a_0(a_0-|c_0|^2)+\sum_{l=2}^{n-1}\sum_{k=1}^{l-1}\overline{c}_{n-l}c_{n-l+k}a_{n-k}+\sum_{l=1}^{n-2}\sum_{k=1}^{n-l-1}\overline{c}_{n-l}c_{n-k-l}a_{k}=\\
\\
\displaystyle =a_0(a_0-|c_0|^2)+\sum_{k=2}^{n-1}a_{k}\sum_{l=n-k+1}^{n-1}\overline{c}_{n-l}c_{2n-l-k}+\sum_{k=1}^{n-2}a_{k}\sum_{l=1}^{n-k-1}\overline{c}_{n-l}c_{n-k-l}=\\
\\
\displaystyle =a_0(a_0-|c_0|^2)+\sum_{k=1}^{n-1}a_{k}\left[\sum_{j=0}^{k-1}\overline{c}_{j}c_{n-k+j}+\sum_{j=1}^{n-k-1}\overline{c}_{j+k}c_{j}\right]-c_0\sum_{k=1}^{n-1}a_{k}\overline{c}_k-
\overline{c}_0\sum_{k=1}^{n-1}a_{k}c_{n-k}=\\
\\
\displaystyle = a_0^2+\sum_{k=1}^{n-1}|a_{k}|^2-a_0|c_0|^2-\sum_{k=1}^{n-1}a_{k}\left[c_0\overline{c}_k+\overline{c}_0c_{n-k}\right]=\\
\\
\displaystyle = e_0-a_0|c_0|^2-\sum_{k=1}^{n-1}a_{k}\left[c_0\overline{c}_k+\overline{c}_0c_{n-k}\right].
\end{array}
\end{equation*}

Analogously, we have
\begin{equation*}
Tr(A_{n-1}\widetilde{D})= e_0-a_0|c_0|^2-\sum_{k=1}^{n-1}a_{k}\left[c_0\overline{c}_k+\overline{c}_0c_{n-k}\right],
\end{equation*}
so
\begin{equation}\label{trace.BB}
Tr(B\widetilde{B})= a_0^2+(n-4)e_0+\left|\sum_{k=1}^{n-1}c_kc_{n-k}\right|^2+2a_0|c_0|^2+2\sum_{k=1}^{n-1}a_{k}\left[c_0\overline{c}_k+\overline{c}_0c_{n-k}\right],
\end{equation}

On another hand, from~\eqref{circulant.eigenvalues} we obtain
\begin{equation}\label{coeff.c_k}
c_k=\dfrac{1}{n}\sum_{j=1}^n\omega_{j-1}^{-k}\lambda_j,\qquad k=1,\ldots,n-1,
\end{equation}
so
\begin{equation*}
\sum_{k=1}^{n-1}c_kc_{n-k}=\dfrac{1}{n^2}\sum_{j=1}^n\sum_{l=1}^n\lambda_j\lambda_l\sum_{k=1}^{n-1}\omega_{j-1}^k\omega_{l-1}^{-k}=
\dfrac{n-1}{n^2}\sum_{j=1}^n\lambda_j^2-\dfrac2n\sum_{1\leqslant l<j\leqslant n}\lambda_j\lambda_l=\dfrac1n\sum_{j=1}^n\lambda_j^2-\dfrac1{n^2}\left(\sum_{j=1}^n\lambda_j\right)^2.
\end{equation*}
Therefore,
\begin{equation}\label{formula.1}
\left|\sum_{k=1}^{n-1}c_kc_{n-k}\right|^2=\left|\dfrac1n\sum_{j=1}^n\lambda_j^2-\dfrac1{n^2}\left(\sum_{j=1}^n\lambda_j\right)^2\right|^2.
\end{equation}
Recall that the eigenvalues of the matrix $A$ defined in~\eqref{circulant.A} are $|\lambda_j|^2$, $j=1,\ldots,n$, so analogously to~\eqref{circulant.eigenvalues},
one has
\begin{equation}\label{circulant.A.eigenvalues}
|\lambda_j|^2=a_0+a_1\omega_{j-1}+a_2\omega_{j-1}^2+\dots+a_{n-1}\omega_{j-1}^{n-1},\quad j=1,\ldots,n.
\end{equation}

Finally, from~\eqref{coeff.c_k} and~\eqref{circulant.A.eigenvalues} we obtain the following
\begin{equation}\label{formula.2}
\begin{array}{c}
\displaystyle a_0|c_0|^2+\sum_{k=1}^{n-1}a_{k}\left[c_0\overline{c}_k+\overline{c}_0c_{n-k}\right]= -a_0|c_0|^2+
\dfrac1{n}\sum_{k=0}^{n-1}a_k\sum_{j=1}^{n}\omega_{j-1}^k\left[\overline{\lambda}_jc_0+\lambda_j\overline{c_0}\right]=\\
\\
\displaystyle=-a_0|c_0|^2+\dfrac1{n}\sum_{j=1}^{n}|\lambda_j|^2\left[\overline{\lambda}_jc_0+\lambda_j\overline{c_0}\right]=
-a_0|c_0|^2+\dfrac1{n}\sum_{j=1}^{n}|\lambda_j|^2\left[|\lambda_j+c_0|^2-|c_0|^2-|\lambda_j|^2\right]=\\
\\
\displaystyle= \dfrac1{n}\sum_{j=1}^{n}|\lambda_j|^2|\lambda_j+c_0|^2-2a_0|c_0|^2-e_0.
\end{array}
\end{equation}
Now the inequality~\eqref{general.shaurma} follows from~\eqref{trace.BB},~\eqref{formula.1},~\eqref{formula.2}, and from~\eqref{coeff.a_0},~\eqref{c_0},~\eqref{e_0}.
\end{proof}

\setcounter{equation}{0}
%%%%%%%%%%%%%%%%%%%%%%%%%%%%%%%%%%%%%%%%%%%%%%%%%%%%%%%%%%%%%%%%%%%%%%%%%%%%%%%%%%%%
\section{Majorization of critical points of polynomials}\label{section:majorization}
%%%%%%%%%%%%%%%%%%%%%%%%%%%%%%%%%%%%%%%%%%%%%%%%%%%%%%%%%%%%%%%%%%%%%%%%%%%%%%%%%%%%

Theorem~\ref{Theorem.circulant.derivative} allows also to obtain some facts on majorization of critical points of polynomials or simplify their proofs. Before we show this, we need to present the definition of majorization.

For any vector $\mathbf{x}=(x_1,\ldots,x_n)$, denote by $(x_{[1]},\ldots,x_{[n]})$ the following rearrangement of the components of $\mathbf{x}$
$$
x_{[1]}\geqslant x_{[2]}\geqslant\cdots\geqslant x_{[n]}.
$$

\begin{definition}\label{def.majorization}
For any two vectors $\mathbf{a}=(a_1,\ldots,a_n)$ and $\mathbf{b}=(b_1,\ldots,b_n)$ from $\mathbb{R}^n$,
we say that $\mathbf{b}$ weakly majorizes $\mathbf{a}$, and denote this as $\mathbf{a}\prec_{w}\mathbf{b}$
if
$$
\sum_{j=1}^ma_{[j]}\leqslant\sum_{j=1}^mb_{[j]}, \qquad m=1,\ldots,n.
$$
Moreover, $\mathbf{a}$ is said to be strongly majorized by $\mathbf{b}$ if, in addition,
for $m=n$, one has
$$
\sum_{j=1}^na_{[j]}=\sum_{j=1}^nb_{[j]}.
$$
\end{definition}

The following fact was established in~\cite{Ky_Fan}.
\begin{theorem}[Ky Fan]
Let $M$ be a complex $n\times n$ matrix, and $\widehat{M}=\frac12(M+M^*)$. Then
$$
(\Re\lambda_1(M),\ldots,\Re\lambda_n(M))\prec_{w}(\lambda_1(\widetilde{M}),\ldots,\lambda_n(\widehat{M})),
$$
where $\lambda_j(M)$ and  $\lambda_j(\Re M)$, $j=1,\ldots,n$, are eigenvalues of the matrices
$M$ and $\widehat{M}$, respectively.
\end{theorem}

Now if $p(z)=\det(zI-C)=\prod_{j=1}^n(z-\lambda_j)$ for some circulant matrix, then numbers
$\Re \lambda_j$, $j=1,\ldots,n$, are the eigenvalues of the matrix $H=\frac12(C+C^*)$ which is a circulant as well.
By Theorem~\ref{Theorem.circulant.derivative}, the eigenvalues $\xi_k$, $k=1,\ldots,n-1$, of the matrix $H_{n-1}=\frac12(C_{n-1}+C^*_{n-1})$
are the critical points of the polynomial $\displaystyle\prod_{j=1}^n(z-\Re\lambda_j)$, and by Ky Fan's theorem one has
\begin{equation}\label{Pereira-Kastorkis}
(\Re w_1,\ldots,\Re w_{n-1})\prec_{w}(\xi_1,\ldots,\xi_{n-1}).
\end{equation}

The inequalities~\eqref{Pereira-Kastorkis} were conjectured by Katsoprinakis~\cite{Katsoprinakis} and
proved by Pereira in~\cite{Pereira}. In fact, this proof is similar to the original proof of Pereira. The
only difference is that we have an explicit differentiator~$C_{n-1}$ which allows us to clarify some ideas.

The following theorem was established in~\cite[Theorem 2.1]{Cheung_Ng_1}.
\begin{theorem}\label{Theorem.Patrick}
Let $p$ be a polynomial of degree $n\geqslant2$ with zeroes $\lambda_j$, $j=1,\ldots,n$, and
critical points $w_k$, $k=1,\ldots,n-1$. Suppose additionally that $\lambda_n=0$. If $q$ be a polynomial with zeroes $|\lambda_j|$, $j=1,\ldots,n$, and
critical points $\xi_k$, $k=1,\ldots,n-1$, then
\begin{equation}\label{Cheung_Ng.majorization}
(\Phi(|w_1|),\ldots,\Phi(|w_{n-1}|))\prec_{w}(\Phi(\xi_1),\ldots,\Phi(\xi_{n-1})),
\end{equation}
for any increasing function $\Phi:[0,+\infty)\mapsto\mathbb{R}$ such that $\Phi\circ\exp$ is convex on $\mathbb{R}$.
\end{theorem}

The circulant approach allows us to avoid the condition $\lambda_n=0$. But we are able only to prove the following close
result.

\vspace{2mm}

\noindent\textbf{Theorem 1.2.}\textit{
Let $p$ be a polynomial of degree $n\geqslant2$ with zeroes $\lambda_j$, $j=1,\ldots,n$, and
critical points $w_k$, $k=1,\ldots,n-1$. If $q$ be a polynomial with zeroes $|\lambda_j|^2$, $j=1,\ldots,n$, and
critical points $\eta_k$, $k=1,\ldots,n-1$, then
\begin{equation}\label{Cheung_Ng.majorization.like.2}
(\Phi(|w_1|),\ldots,\Phi(|w_{n-1}|))\prec_{w}(\Phi(\sqrt{\eta_1}),\ldots,\Phi(\sqrt{\eta_{n-1}})),
\end{equation}
for any increasing function $\Phi:[0,+\infty)\mapsto\mathbb{R}$ such that $\Phi\circ\exp$ is convex on $\mathbb{R}$.
}
\begin{proof}
For certainty, let us fix the arrangement of the numbers $\lambda_j$, $j=1,\ldots,n$, say, as in~\eqref{order.EVal},
and construct the unique circulant $C$ with the spectrum $\{\lambda_1,\ldots,\lambda_{n}\}$.
By Proposition~\ref{Proposition.abs.val}, the unique circulant matrix with eigenvalues $|\lambda_j|^2$, $j=1,\ldots,n$,
is $A=CC^*$. By Theorem~\ref{Theorem.circulant.derivative}, the spectrum of the matrix $A_{n-1}$
is $\{\xi_1,\ldots,\xi_{n-1}\}$, while the spectrum of $C_{n-1}$ is $\{w_1,\ldots,w_{n-1}\}$.
By construction, we have $A_{n-1}=C_{n-1}C^*_{n-1}$, so the numbers $\sqrt{\eta_k}$, $k=1,\ldots,n-1$, are the singular values (see
e.g.~\cite[p.~135]{Horn_Johnson}) of matrix $C_{n-1}$. If now we arrange the numbers $w_k$ and~$\eta_k$, $k=1,\ldots,n-1$,
as follows
$$
|w_1|\geqslant|w_2|\geqslant\cdots\geqslant|w_{n-1}|\qquad\text{and}\qquad|\eta_1|\geqslant|\eta_2|\geqslant\cdots\geqslant|\eta_{n-1}|,
$$
that is always possible, then \cite[Theorem~3.3.13]{Horn_Johnson} implies~\eqref{Cheung_Ng.majorization.like.2}.
\end{proof}

Note that in a similar way, one can obtain that (Theorem~\ref{Theorem.critical.squares}) for any increasing function $\Phi:[0,+\infty)\mapsto\mathbb{R}$ such that $\Phi\circ\exp$ is convex on $\mathbb{R}$,
\begin{equation*}\label{crit.point.majorization}
(\Phi(\xi_1),\ldots,\Phi(\xi_{n-1}))\prec_{w}(\Phi(\sqrt{\eta_1}),\ldots,\Phi(\sqrt{\eta_{n-1}})),
\end{equation*}
where $\xi_k$ and $\eta_k$ are the critical points of the polynomial $\displaystyle\prod_{j=1}^n(z-|\lambda_j|)$ and
$\displaystyle\prod_{j=1}^n(z-|\lambda_j|^2)$, respectively, so in the case $\lambda_n=0$, Theorem~\ref{Theorem.Patrick} implies
Theorem~\ref{Theorem.Cheung_Ng.majorization}. However, we do not see how to use circulants to prove and improve Theorem~\ref{Theorem.Patrick}.

%\setcounter{equation}{0}
%%%%%%%%%%%%%%%%%%%%%%%%%%%%%%%%%%%%%%%%%%%%%%%%%%%%%%%%%%%%%%%%%%%%%%%%%%%%%%%%%%%%
%\section{Higher order Schoenberg conjectures}
%%%%%%%%%%%%%%%%%%%%%%%%%%%%%%%%%%%%%%%%%%%%%%%%%%%%%%%%%%%%%%%%%%%%%%%%%%%%%%%%%%%%

\setcounter{equation}{0}
%%%%%%%%%%%%%%%%%%%%%%%%%%%%%%%%%%%%%%%%%%%%%%%%%%%%%%%%%%%%%%%%%%%%%%%%%%%%%%%%%%%%
\section{Questions and open problems}\label{Section:open.questions}
%%%%%%%%%%%%%%%%%%%%%%%%%%%%%%%%%%%%%%%%%%%%%%%%%%%%%%%%%%%%%%%%%%%%%%%%%%%%%%%%%%%%

In this Section, we enumerate some problems related to circulants and root location of
polynomials.

\begin{itemize}
\item[1)] If $p(z)$ is a characteristic polynomial of a circulant matrix $C$, then
by Schur's theorem, for the critical points $w_k$ of $p(z)$, one has
\begin{equation*}
\sum_{k=1}^{n-1}|w_k|^{2m}\leqslant Tr\left(C^m_{n-1}(C_{n-1}^*)^m\right).
\end{equation*}
However, even for $m=3$ it is rather a technically difficult problem to express $Tr\left(C^m_{n-1}(C_{n-1}^*)^m\right)$
via roots of the polynomial $p(z)$. We are sure that, anyway, the case $m=3$ can be covered at least for the
situation when the sum of roots of $p(z)$ equals zero $(c_0=0)$. But we postpone solution of this problem,
since so far, we do not see that anything new can be found here from methodological point of view. Nevertheless,
estimate for $\sum_{k=1}^{n-1}|w_k|^{6}$ (stronger than the one provided by Schmeisser's theorem) would be quite
interesting.

 In view of the theory presented in~\cite{Bark_thesis,Bark_lectures,Pereira,Malamud,Cheung_Ng_1,Cheung_Ng_2} and in the present paper, it is interesting to know wether it is possible to estimate
$$\displaystyle \sum_{k=1}^{n-1}|w_k|^{2m+1},$$
for $m=0,1,2,\ldots$ In this case, Proposition~\ref{Proposition.abs.val} may be of use.
\item [2)] Let again $p(z)$ be a monic polynomial and $C$ is a corresponding circulant matrix such that $p(z)=\det(zI-C)$.
Then for any $\alpha\in\mathbb{C}$, the polynomial $p(z)+\alpha p'(z)$ is the characteristic polynomial of the following
matrix
\begin{equation*}
\begin{pmatrix}
c_0-\alpha    &c_1    &c_2   &\dots &c_{n-2}&c_{n-1}\\
c_{n-1}&c_0    &c_1   &\dots &c_{n-3}&c_{n-2}\\
c_{n-2}&c_{n-1}&c_0   &\dots &c_{n-4}&c_{n-3}\\
\vdots &\vdots &\vdots&\ddots&\vdots &\vdots \\
c_2    &c_3    &c_4   &\dots &c_0    &c_1    \\
c_1    &c_2    &c_3   &\dots &c_{n-1}&c_0    \\
\end{pmatrix}
\end{equation*}
Thus, this matrix can be used to study and estimates the roots
of the polynomial $p(z)+\alpha p'(z)$. For example, it might be
used to attack Conjecture 1 in~\cite{Borcea_Shapiro}.
\item[3)] The relations between the eigenvalues of a circulant $C$ and its entries $c_k$ are well known. It would be interesting to find relations between $c_k$ and the critical points of the polynomial $\det(zI-C)$. The eigenvectors of the matrix $C_{n-1}$ are of particular interest.
\end{itemize}

%%%%%%%%%%%%%%%%%%%%%%%%%%%%%%%%%%%%%%%%%%%%%%%%%%%%%%%%%%%%%%%%%%%%%%%%%%%%%%%%%%%%
\section*{Acknowledgements}
The work of the first author was partially supported by 
National Science Foundation of China, grant no. BC0710062. The second author is Shanghai Oriental Scholar
whose work was supported by Russian Science Foundation, grant no. 14-11-00022.
%%%%%%%%%%%%%%%%%%%%%%%%%%%%%%%%%%%%%%%%%%%%%%%%%%%%%%%%%%%%%%%%%%%%%%%%%%%%%%%%%%%%

\end{document}